\documentclass{amsart}
\usepackage{amsmath,amsfonts,amssymb,enumerate}

\newcommand{\N}{\mathbb{N}}                   
\newcommand{\R}{\mathbb{R}}                   
\newcommand{\C}{\mathbb{C}}                   
\newcommand{\AR}{\mathcal{AR}}
\newcommand{\Reg}{\mathrm{Reg}}

\theoremstyle{plain}
\newtheorem{theorem}{Theorem}[section]
\newtheorem{proposition}[theorem]{Proposition}
\newtheorem{lemma}[theorem]{Lemma}
\newtheorem{corollary}[theorem]{Corollary}

\theoremstyle{definition}

\newtheorem{example}[theorem]{Example}
\newtheorem{remark}[theorem]{Remark}

\numberwithin{equation}{section}

\begin{document}
\title{CR-continuation of arc-analytic maps}

\author{Janusz Adamus}
\address{Department of Mathematics, The University of Western Ontario, London, Ontario, Canada N6A 5B7 -- and --
         Institute of Mathematics, Polish Academy of Sciences, ul. {\'S}niadeckich 8, 00-956 Warsaw, Poland}
\email{jadamus@uwo.ca}
\thanks{Research was partially supported by Natural Sciences and Engineering Research Council of Canada.}

\begin{abstract}
Given a set $E$ in $\C^m$ and a point $p\in E$, there is a unique smallest complex-analytic germ $X_p$ containing $E_p$, called the holomorphic closure of $E_p$. We study the holomorphic closure of semialgebraic arc-symmetric sets. Our main application concerns CR-continuation of semialgebraic arc-analytic mappings: A mapping $f:M\to\C^n$ on a connected real-analytic CR manifold which is semialgebraic arc-analytic and CR on a non-empty open subset of $M$ is CR on the whole $M$.
\end{abstract}

\maketitle


\section{Introduction}
\label{sec:intro}

Let $\Omega$ be an open set in $\R^m$. A function $f$ real-analytic in $\Omega$ is called \emph{Nash} if there is a non-constant polynomial $P\in\R[x,y]$, where $x=(x_1,\dots,x_m)$, such that $P(x,f(x))=0$ for all $x\in\Omega$. A real-analytic set is a \emph{Nash set} if it is (locally) defined by Nash real-analytic functions.

In the present note we will be concerned with a more general class of functions, which appear naturally in real geometry. Given a real-analytic set $R$ in an open $\Omega\subset\R^m$, a function $f:R\to\R$ is called \emph{arc-analytic} if it is analytic on every arc, that is, if $f\circ\gamma$ is analytic for every real-analytic $\gamma:(-\varepsilon,\varepsilon)\to R$.

Recall that a set $E$ in $\R^m$ is called \emph{semialgebraic} if it is a finite union of sets of the form
\[
\{x=(x_1,\dots,x_m)\in\R^m:f_1(x)=\dots=f_r(x)=0, g_1(x)>0,\dots,g_s(x)>0\}\,,
\]
where $r,s\in\N$ and $f_1,\dots,f_r,g_1,\dots,g_s\in\R[x]$. A function $f:E\to\R$ is called semialgebraic if its graph $\Gamma_f$ is a semialgebraic subset of $\R^m\times\R$. A \emph{semialgebraic mapping} $f=(f_1,\dots,f_n):E\to\R^n$ is one all of whose components $f_j$ are semialgebraic. Identifying $\C^m$ with $\R^{2m}$, one can speak of real-analytic, Nash, and semialgebraic subsets of $\C^m$.

Our main result is the following variant of a theorem of Shafikov \cite{Sh} on CR-continuation of continuous mappings (for details on CR structure and CR functions see Section~\ref{sec:prelim}).

\begin{theorem}
\label{thm:cr-continuation}
Let $M$ be a connected Nash real-analytic CR submanifold of an open set in $\C^m$ and let $f:M\to\C^n$ be a semialgebraic arc-analytic mapping ($m,n\geq1$). If $f$ is CR on a nonempty open subset of $M$, then $f$ is CR on $M$.
\end{theorem}

Semialgebraic arc-analytic mappings are necessarily continuous (\cite[Prop.\,5.1]{Kur}).
However, an arc-analytic map need not have a real-analytic graph (see Example~\ref{ex:arc-analytic} below). It follows that Theorem~\ref{thm:cr-continuation} is strictly stronger than Shafikov's \cite[Thm.\,1.3]{Sh} in the Nash setting.

Note also that semialgebraicity itself would not suffice to obtain the conclusion of Theorem~\ref{thm:cr-continuation}. Indeed, a continuous semialgebraic map $f:M\to\C^n$ may have very different CR properties on different open subsets of $M$. Consider, for instance, $M=\C$ and a function $f:M\to\C$ defined as $f(x+iy)=\sqrt{x^4+y^4}$ for $x^4+y^4\leq1$ and $f(x+iy)=1$ for $x^4+y^4\geq1$. Then $f$ is holomorphic in some open neighbourhood of every $z_0$ with, say, $|z_0|>2$, but $f$ is not CR on $M$.
\smallskip

Arc-analytic functions, although relatively unknown among non-specialists, form a very important class in real-analytic geometry. Indeed, Bierstone and Milman (\cite{BM}) proved that arc-analytic semialgebraic functions on a Nash manifold are precisely those that can be made Nash analytic after composition with a finite sequence of blowings-up with smooth nowhere-dense centers. Many classical examples in calculus are arc-analytic but not analytic.

\begin{example}
\label{ex:arc-analytic} (a) The function $f:\R^2\to\R$ defined as $f(x,y)=x^3/(x^2+y^2)$ for $(x,y)\neq(0,0)$ and $f(0,0)=0$ is arc-analytic but not differentiable at the origin. (Observe that $f$ is made Nash after composition with a single blowing-up of the origin; for instance, $f(x,xy)=x/(1+y^2)$.) Note also that the graph $\Gamma_f$ of $f$ is not real-analytic. In fact, the smallest real-analytic subset of $\R^3$ containing $\Gamma_f$ is the \emph{Cartan umbrella} $\{(x,y,z)\in\R^3:z(x^2+y^2)=x^3\}$ (cf. \cite[Ex.\,1.2(1)]{Kur}).

(b) The function $g:\R^2\to\R$ defined as $g(x,y)=\sqrt{x^4+y^4}$ is arc-analytic but not $\mathcal{C}^2$. The graph $\Gamma_g$ of $g$ is not real-analytic. Indeed, the Zariski closure $\{(x,y,z)\in\R^3:z^2=x^4+y^4\}$ of $\Gamma_g$ has two $\mathcal{C}^1$ sheets $z=\pm\sqrt{x^4+y^4}$, but it is irreducible at the origin as a real-analytic set (cf. \cite[Ex.\,1.2(3)]{BM}).
\end{example}
\medskip

The main tool used in the proof of Theorem~\ref{thm:cr-continuation} is the so-called holomorphic closure (HC for short).
Given a set $E$ in $\C^m$ and a point $p\in E$, one defines the \emph{holomorphic closure} of the germ $E_p$ as the unique smallest (with respect to inclusion) complex analytic germ which contains $E_p$; denoted $\overline{E_p}^{HC}$. The \emph{holomorphic closure dimension} of $E_p$, $\dim_{HC}E_p$, is the (complex) dimension of $\overline{E_p}^{HC}$.
Holomorphic closure of real-analytic germs in complex spaces had been studied in \cite{Sh}, \cite{AS} and \cite{AR}. It is closely connected with the CR structure (see, e.g., Prop.\,1.4 and Thm.\,1.5 of \cite{AS}). In \cite{AR}, we considered holomorphic closure in the semialgebraic category. We showed there that HC dimension is tame in this category, which was used to prove the existence of semialgebraic stratification by CR manifolds.

In the present paper we continue the study of the HC structure of semialgebraic sets.
This time we investigate how the HC dimension behaves on arc-symmetric sets. Recall that a set $E\subset\R^m$ is called \emph{arc-symmetric} when, for every real-analytic arc $\gamma:(-1,1)\to\R^m$ with $\gamma((-1,0))\subset E$, there exists $\varepsilon>0$ such that $\gamma((-1,\varepsilon))\subset E$ (cf. \cite[Def.\,1.1]{Kur}).
The concept of arc-symmetry was introduced by Kurdyka \cite{Kur} in the semialgebraic category. It allows one to make sense of the notions of irreducibility and components of a semialgebraic set much like in the algebraic case (see Section~\ref{sec:prelim} for details).

Semialgebraic arc-symmetric subsets of $\R^m$ will be called $\AR$-closed sets (cf. Theorem~\ref{thm:AR-topology}). The following result lies at the center of our arguments.

\begin{theorem}
\label{thm:ar-irred-const-hcdim}
Let $E\subset\C^m$ be a semialgebraic $\AR$-irreducible set of pure dimension. Then:
\begin{itemize}
\item[(i)] There exists an integer $h$ such that $\dim_{HC}E_p=h$ for all $p\in E$.
\item[(ii)] If $Z$ is the smallest complex-algebraic set in $\C^m$ containing $E$, then $Z$ is irreducible and of (complex) dimension $h$.
\end{itemize}
\end{theorem}

Note that the set $Z$ in the above theorem need not realize the HC closure of $E$ at each of its points (see, e.g., \cite[\S\,2]{Sh} or \cite[Ex.\,4.4]{AB}). It does so, however, when $E$ has a complex-analytic germ at some point.
The following result is an arc-symmetric analogue of Shafikov's \cite[Cor.\,1.2]{Sh}.
It will be used in the proof of Theorem~\ref{thm:cr-continuation}.

\begin{theorem}
\label{thm:ar-irred-is-C-alg}
Let $E\subset\C^m$ be a semialgebraic $\AR$-irreducible set of pure dimension. If $E$ contains a point $p$ such that $E_p$ is a complex-analytic germ, then $E$ is irreducible complex-algebraic.
\end{theorem}

In the next section we recall basic definitions and tools used in this article. Theorems~\ref{thm:ar-irred-const-hcdim} and~\ref{thm:ar-irred-is-C-alg} are proved in Section~\ref{sec:holoclos-vs-arcsym}, which is devoted to the study of holomorphic closure of arc-symmetric sets. The last section contains the proof of Theorem~\ref{thm:cr-continuation}.


\section{Preliminaries}
\label{sec:prelim}

\subsection{CR manifolds and CR functions}

Given an $\R$-linear subspace $L$ in $\C^m$ of dimension $d$, one defines the {\it CR dimension} of $L$ to be the largest $k$ 
such that $L$ contains a $\C$-linear subspace of (complex) dimension $k$. Clearly, $0\le k\le \left[\frac{d}{2}\right]$.
A $d$-dimensional real-analytic submanifold $M$ of an open set in $\C^m$ is called a {\it CR manifold} of CR dimension $k$, if the tangent space $T_p M$ has CR dimension $k$ for every point $p\in M$ (the $k$-dimensional complex vector subspace of $T_p M$ will be then denoted by $H_pM$). We write 
$\dim_{CR}M=k$. If $k=0$, then $M$ is called a {\it totally real} submanifold.
The integer $l:=d-2k$ is called the \emph{CR codimension} of $M$, and the pair $(k,l)$ is the \emph{type} of $M$. A CR submanifold $M$ in $\C^m$ is called \emph{generic} when $m=k+l$, where $M$ is of type $(k,l)$.

The notion of \emph{CR function} is usually defined in terms of tangential Cauchy-Riemann equations, as follows. Given a real-analytic CR submanifold $M$ in an open set in $\C^m$, a smooth vector field $X$ on $M$ is called a \emph{CR vector field} if $X_p\in H_pM$ for every $p\in M$. A $\mathcal{C}^1$-smooth function $f:M\to\C$ is CR if $Xf\equiv0$ for every CR vector field $X$ on $M$ of the form $X=\sum_{j=1}^mc_j\frac{\partial}{\partial\bar{z}_j}$.
However, in order to use Theorem~\ref{thm:cr-continuation} in its full generality, we shall need a more general definition that does not require smoothness of $f$. This can be done in terms of distributions: Suppose that $M$ is of type $(k,l)$. A locally integrable function $f:M\to\C$ is called a \emph{CR function} if
\[
\int_Mf\overline{\partial}\alpha=0
\]
for any differential form $\alpha$ of bidegree $(k,k+l-1)$ with compact support (cf. \cite[\S\,21]{Sha}).
A \emph{CR mapping} $f=(f_1,\dots,f_n):M\to\C^n$ is one all of whose components $f_j$ are CR functions.

\subsection{Holomorphic closure of a semialgebraic set germ}

Given a set $E$ in $\C^m$ and a point $p\in E$, one defines the \emph{holomorphic closure} of the germ $E_p$ as the unique smallest (with respect to inclusion) complex analytic germ which contains $E_p$; denoted $\overline{E_p}^{HC}$. The \emph{holomorphic closure dimension} of $E_p$, $\dim_{HC}E_p$, is the (complex) dimension of $\overline{E_p}^{HC}$.
For $d\in\N$, we denote by $\mathcal{S}^d(E)$ the set of those points $p\in E$ for which $\dim_{HC}E_p\geq d$.

For the reader's convenience, we recall the following results from \cite{AS} and \cite{AR} that will be used in the present paper.

\begin{theorem}[{cf. \cite[Thm.\,1.3]{AS}}]
\label{thm:AS-thm3}
Let $E\subset\C^m$ be a connected semialgebraic set of pure dimension, contained in an irreducible real-analytic set of the same dimension. Then the holomorphic closure dimension is constant on $E$.
\end{theorem}

\begin{lemma}[{\cite[Lem.\,2]{AR}}]
\label{lem:AR-lem2}
Let $E$ be a connected real-analytic submanifold of an open set in $\C^m$. There exists a unique smallest complex-algebraic subset $Z$ of $\C^m$ containing $\overline{E}$ and such that, for every $p\in\overline{E}$, $Z_p$ is the smallest complex-algebraic germ containing $E_p$. Moreover, $Z$ is irreducible.
\end{lemma}

\begin{proposition}[{\cite[Prop.\,5]{AR}}]
\label{prop:AR-prop5}
Let $E$ be a connected semialgebraic real-analytic submanifold of an open subset of $\C^m$, and let $Z$ be the unique irreducible complex-algebraic set from the above lemma. Then, at every point $p\in\overline{E}$, the holomorphic closure $\overline{E_p}^{HC}$ is a union of some analytic-irreducible components of $Z_p$. In particular, the holomorphic closure dimension is constant on $E$.
\end{proposition}

\subsection{Nash real-analytic sets}

There is a little ambiguity in literature regarding (real) Nash sets. In fact, there are two common definitions, which do not coincide in general. According to one definition, a \emph{Nash set} is a real-analytic subset $R$ of an open set $\Omega$ in $\R^m$ with the property that for every $p\in\Omega$ there is an open neighbourhood $U$ of $p$ in $\Omega$ such that $R\cap U$ is defined by the vanishing of finitely many functions that are Nash analytic in $U$ (as in Section~\ref{sec:intro}). This is how real-analytic Nash sets are understood, for instance, in \cite{BER}. Another classical monograph, \cite{BCR} defines $R\subset\R^m$ to be Nash when it is real-analytic in an open $\Omega$ in $\R^m$ and a semialgebraic subset of $\R^m$.

\begin{remark}
\label{rem:two-Nashes}
It is not difficult to see that a set $R$ which is Nash in the second sense is also Nash in the first sense. The two notions coincide when $R$ is compact or, more generally, when $R$ admits a finite partition into connected smooth manifolds. (They also always coincide in the complex-analytic case.) However, in general, a set which is Nash in the first sense need not be semialgebraic (for example, its regular locus may have an infinite number of connected components).
\end{remark}

As far as Theorem~\ref{thm:cr-continuation} goes, one can disregard the above ambiguity, because the domain of a semialgebraic mapping $f$ is always semialgebraic (by the Tarski-Seidenberg Theorem (see, e.g., \cite[Prop.\,2.2.7]{BCR}), as a projection of the semialgebraic set $\Gamma_f$).

\subsection{Semialgebraic arc-symmetric sets}

The notion of arc-symmetry was introduced by Kurdyka \cite{Kur} in the semialgebraic category. Its usefulness comes from the following fundamental result.

\begin{theorem}[{\cite[Thm.\,1.4]{Kur}}]
\label{thm:AR-topology}
There exists a unique noetherian topology $\AR$ on $\R^m$, such that the $\AR$-closed sets are precisely the semialgebraic arc-symmetric sets in $\R^m$.
\end{theorem}

(A topology is called \emph{noetherian}, when every descending sequence of closed sets is stationary.)
For a detailed exposition of arc-symmetric sets, we refer the reader to \cite{Kur} or \cite{KP}. Here we recall only a few basic properties that will be used in the present paper.

The class of $\AR$-closed sets includes, in particular, the real-algebraic sets as well as the Nash real-analytic sets (in the sense of \cite{BCR}). The $\AR$-topology is strictly finer than the Zariski topology on $\R^m$ (see, e.g., \cite[Ex.\,1.2]{Kur}). Moreover, it follows from the semialgebraic Curve Selection Lemma that $\AR$-closed sets are closed in the Euclidean topology in $\R^m$ (see \cite[Rem.\,1.3]{Kur}).

An $\AR$-closed set $E$ is called \emph{$\AR$-irreducible}, when $E=E_1\cup E_2$ with $E_1,E_2$ $\AR$-closed implies that $E=E_1$ or $E=E_2$. It follows from Theorem~\ref{thm:AR-topology} that every $\AR$-closed set $E$ admits a (unique) finite decomposition $E=E_1\cup\dots\cup E_s$ into $\AR$-irreducible subsets such that $E_j\not\subset\bigcup_{k\neq j}E_k$ for $j=1,\dots,s$. The sets $E_j$ of this decomposition are called the $\AR$-irreducible components of $E$.

In the next section, we will study the HC structure of $\AR$-closed sets in a complex space. These sets will play an important role in the proof of Theorem~\ref{thm:cr-continuation}, because, according to \cite[Prop.\,5.1]{Kur}, the graph of a semialgebraic arc-analytic mapping is an $\AR$-closed set.


\section{Holomorphic closure of an arc-symmetric set}
\label{sec:holoclos-vs-arcsym}

In general, $\AR$-closed sets (even $\AR$-irreducible) need not be of pure holomorphic closure dimension. Even worse than that, the HC-filtration $\{\mathcal{S}^d(E)\}_{d\in\N}$ of an $\AR$-closed set $E$ need not be $\AR$-closed itself. For example, the \emph{Whitney umbrella} embedded in $\C^2$ as
\[
V=\{(x_1+iy_1,x_2+iy_2)\in\C^2:x_1^2-x_2y_1^2=y_2=0\}
\]
happens to be $\AR$-irreducible (cf. \cite[Rem.\,3.2]{Kur}), but $\mathcal{S}^2(V)$ coincides with the two-dimensional part of $V$ (since the 'stick' $\{x_1=y_1=y_2=0\}$ is contained in a complex line), which is not $\AR$-closed.

Nonetheless, according to Theorem~\ref{thm:ar-irred-const-hcdim}, the above problems can occur only in the mixed-dimensional case.

\subsection{Proof of Theorem~\ref{thm:ar-irred-const-hcdim}}
\label{subsec:proof2}

The proof of Theorem~\ref{thm:ar-irred-const-hcdim} will rely on the following two lemmas.

\begin{lemma}
\label{lem:C-zar-clos-irred}
Let $E$ be an $\AR$-irreducible set in $\C^m$. If $Z$ is the smallest (with respect to inclusion) complex-algebraic set containing $E$, then $Z$ is irreducible.
\end{lemma}

\begin{proof}
Suppose that $Z=Z_1\cup Z_2$ is a union of two complex-algebraic sets, with $Z_j\neq Z$, $j=1,2$. The sets $E\cap Z_1$, $E\cap Z_2$ being semialgebraic and arc-symmetric, it follows that $E\subset Z_1$ or $E\subset Z_2$ (by $\AR$-irreducibility of $E$). This contradicts the minimality of $Z$.
\end{proof}

Given a real-analytic or semialgebraic set $E$ and $d\in\N$, we denote by $\mathrm{Reg}_dE$ the locus of points $p\in E$ such that $E_p$ is a germ of a $d$-dimensional manifold.

\begin{lemma}
\label{lem:an-clos-irred}
Let $E$ be a $k$-dimensional $\AR$-irreducible set in $\C^m$.
If $R$ is the smallest real-analytic set containing $\overline{\Reg_k(E)}$, then $R$ is irreducible and of dimension $k$.
\end{lemma}

\begin{proof}
Observe first that $R$ is well-defined, since the intersection of the family of all real-analytic sets in $\C^m$ containing $\overline{\Reg_k(E)}$ is itself real-analytic (see, e.g., \cite[Ch.\,V, \S\,2, Cor.\,2]{Nar}). Since the Zariski closure of the semialgebraic set $\overline{\Reg_k(E)}$ (i.e., the smallest real-algebraic set containing $\overline{\Reg_k(E)}$\,) is $k$-dimensional (see \cite[\S\,2.8]{BCR}), it follows that $\dim{R}=k$.

Suppose then that $R=R_1\cup R_2$ is a union of two real-analytic sets, with $R_j\neq R$, $j=1,2$.
Since $\Reg_k(E)\subset R_1\cup R_2$, either $R_1$ or $R_2$ must contain a nonempty open subset of $\Reg_k(E)$; say, the former. Let $\Reg_k(E)=E_1\cup\dots\cup E_s$ be the decomposition into (finitely many, by semialgebraic stratification) connected components. After renumbering the components if needed, we get that $R_1$ contains a nonempty open subset of $E_1$, and hence $E_1\subset R_1$, by the Identity Principle. By closedness of $R_1$, we get $\overline{E}_1\subset R_1$.

Now, after renumbering the components if needed, we can assume that $\overline{E}_1,\dots,$ $\overline{E}_t\subset R_1$ and $\overline{E}_{t+1},\dots,\overline{E}_s\not\subset R_1$, for some $t\leq s$. If $t=s$, then $\overline{\Reg_k(E)}=\overline{E}_1\cup\dots\cup\overline{E}_s\subset R_1$, which contradicts the choice of $R$. Hence $t<s$, and so $E_s\not\subset R_1$.
Choose arbitrary $p\in E_s\setminus R_1$ and $q\in E_1$.
By $\AR$-irreducibility of $E$ and \cite[Cor.\,2.8]{Kur}, there is an arc $\gamma:[0,1]\to\overline{\Reg_k(E)}$ analytic in a neighbourhood of $[0,1]$ and such that $\gamma(0)=q$ and $\gamma(1)=p$. Since $q\in E_1$, we have $\mathrm{Int}(\gamma^{-1}(R_1))\neq\varnothing$. But $R_1$ is arc-symmetric (as a real-analytic set), and so $\gamma([0,1])\subset R_1$. In particular, $p\in R_1$; a contradiction. This shows that $R$ is irreducible.
\end{proof}

\subsubsection*{Proof of Theorem~\ref{thm:ar-irred-const-hcdim}}
If $E$ is of pure dimension $k$, then $E=\overline{\Reg_k(E)}$. Hence, by Lemma~\ref{lem:an-clos-irred}, the smallest real-analytic set $R$ containing $E$ is irreducible and of dimension $k$. The claim (i) thus follows from Theorem~\ref{thm:AS-thm3}.

For the proof of (ii), let $\Reg_k(E)=E_1\cup\dots\cup E_s$ be the decomposition into connected components. By Proposition~\ref{prop:AR-prop5}, for every $j$, there is a unique smallest complex-algebraic set $Z_j$ such that $\overline{(E_j)_p}^{HC}$ is a union of certain analytic-irreducible components of $(Z_j)_p$ for all $p\in E_j$. By part (i) of the theorem, $\dim Z_j=h$ for $j=1,\dots,s$, and so $E=\overline{\Reg_k(E)}$ is contained in an $h$-dimensional complex-algebraic set $Z_1\cup\dots\cup Z_s$.
Now, let $Z$ be the smallest complex-algebraic set that contains $E$. By Lemma~\ref{lem:C-zar-clos-irred}, $Z$ is irreducible. Hence $Z$ is of pure dimension and, clearly, $\dim{Z}\leq\dim(Z_1\cup\dots\cup Z_s)=h$. On the other hand, $\dim Z\geq h$, for otherwise the constant HC-dimension of $E$ would be less than $h$.
\qed
\medskip

\subsection{Corollaries of Theorem~\ref{thm:ar-irred-const-hcdim}}
\label{subsec:postproof2}

We will now derive two simple consequences of Theorem~\ref{thm:ar-irred-const-hcdim}.

\subsubsection*{Proof of Theorem~\ref{thm:ar-irred-is-C-alg}}
Fix $p\in E$ \ and a semialgebraic open neighbourhood $U$ of $p$ in $\C^m$ such that $E\cap U$ is a complex-analytic subset of $U$. Then $E\cap U$ is a complex Nash subset of $U$, by \cite[Prop.\,3]{AR}. Let $X$ be an analytic irreducible component of $E\cap U$ of maximal dimension. 
By \cite[Thm.\,2.10]{Tw}, there is an irreducible complex-algebraic set $Y$ in $\C^m$ such that $X$ is an irreducible component of $Y\cap U$.

Let $Z$ be the unique smallest complex-algebraic set containing $E$. Set $k:=\dim E$.
Since $E\cap U$ is complex-analytic, we get that $\dim_\C X_p=\dim_{HC}E_p=1/2\dim_\R E_p$. Hence $X$ is of pure dimension $h:=k/2$, and $h$ is the constant holomorphic closure dimension of $E$. By Theorem~\ref{thm:ar-irred-const-hcdim}(ii), $\dim Z=h$. By irreducibility of $Y$, we also have that $\dim Y=\dim X=h$. Since $X\subset Z\cap Y$, we have $\dim(Z\cap Y)\geq\dim X=h$, and hence $Z=Y$, by irreducibility.

It follows that $E$ contains a nonempty open subset of $Z$, and so $\dim(\Reg{Z}\cap E)=\dim\Reg{Z}$. By irreducibility of $Z$, $\Reg{Z}$ is a connected (semialgebraic) smooth manifold (cf. \cite[Ch.\,VII,\S\,11.1]{Loj}), and thus $\Reg{Z}\subset E$, by \cite[Rem.\,1.6]{Kur}.
Therefore, $Z=\overline{\Reg{Z}}\subset E\subset Z$, so $E=Z$ is irreducible complex-algebraic, as required.
\qed
\medskip

Another consequence of Theorem \ref{thm:ar-irred-const-hcdim} is the following observation (cf. example at the beginning of this section).

\begin{corollary}
\label{cor:hc-filtration-is-ar}
Let $E$ be an arc-symmetric set in $\C^m$. If $E$ is of pure dimension, then its HC-filtration is $\AR$-closed, that is, $\mathcal{S}^d(E)$ is $\AR$-closed for every $d\in\N$.
\end{corollary}

\begin{proof}
Let $E=E_1\cup\dots\cup E_s$ be the decomposition of $E$ into $\AR$-irreducible components. Then each $E_j$ is of pure dimension, and hence of pure HC dimension, by Theorem~\ref{thm:ar-irred-const-hcdim}. It follows that, for a given $d\in\N$, $\mathcal{S}^d(E)$ is the union of those $E_j$ whose HC dimension is at least $d$. Therefore $\mathcal{S}^d(E)$ is $\AR$-closed as a finite union of $\AR$-closed sets.
\end{proof}

\begin{remark}
\label{rem:not-so-for-inner-dim}
The above corollary has no analogue for the (inner) complex dimension. That is, a pure-dimensional semialgebraic arc-symmetric set need not have $\AR$-closed filtration by the complex dimension. We will use the following notation from \cite{ARS}: For a real analytic set $R$ in $\C^m$ and $d\in\N$, let $\mathcal{A}^d(R)$ denote the set of points $p\in R$ such that $R_p$ contains a complex analytic germ of (complex) dimension $d$.
Now, the irreducible real-algebraic hypersurface
\[
X=\{(z_1,\dots,z_4)\in\C^4: x_1^2-x_2^2+x_3^2=x_4^3\}\,,
\]
where $z_j=x_j+iy_j$, satisfies $\mathcal{A}^1(X)=X\cap\{x_4\geq0\}$ near the point $p=(1,1,0,0)$. (We know from \cite[Thm.\,1.1]{ARS} that $\mathcal{A}^1(X)$ is a closed semialgebraic subset of $X$.) But $X$ is a $7$-manifold near $p$ (in fact, everywhere except at the origin), hence $\mathcal{A}^1(X)$ contains a nonempty open subset of $\Reg_7X=X\setminus\{0\}$. By Kurdyka's \cite[Rem.\,1.6]{Kur}, if $\mathcal{A}^1(X)$ were arc-symmetric it would need to contain the whole $X\setminus\{0\}$.
\end{remark}
\medskip

\subsection{HC dimension on a Nash real-analytic set}
\label{subsec:HC-on-Nash}

Independently of the above, one can show that an irreducible real-algebraic set of pure dimension has constant HC dimension. This is a consequence of the following proposition.

\begin{proposition}
\label{prop:alg-constant-hcdim}
Let $X$ be an irreducible $k$-dimensional real-algebraic set in $\C^m$. Then $\overline{\Reg_k(X)}$ is of constant HC dimension.
Moreover, there is a unique smallest pure-dimensional complex-algebraic set $Z$ in $\C^m$ such that $X\subset Z$ and $\dim_\C Z$ is equal to the constant HC dimension of $\overline{\Reg_k(X)}$.
\end{proposition}

\begin{proof}
Set $Y:=\overline{\Reg_k(X)}$.
Let $d:=\max\{\dim_{HC}Y_p:p\in Y\}$, and suppose that $\mathcal{S}^d(Y)$ is a proper subset of $Y$. Let $E_1,\dots,E_s$ be the connected components of $\Reg_k(X)$. Then $Y=\overline{E}_1\cup\dots\cup\overline{E}_s$. By Proposition~\ref{prop:AR-prop5}, for each $j=1,\dots,s$, there is a unique smallest irreducible complex-algebraic set $Z_j$ in $\C^m$ such that $\overline{E}_j\subset Z_j$ and $\dim_\C Z_j=\dim_{HC}(E_j)_p$ for all $p\in E_j$. By our hypothesis, there exists $j_0$ for which $\dim_\C Z_{j_0}<d$.
Set
\[
\tilde{Z}:=\bigcup_{\dim Z_j<d}Z_j\,.
\]
Then $\dim_\C\tilde{Z}<d$, and hence $Y\not\subset\tilde{Z}$. On the other hand, $\tilde{Z}\cap X\supset E_{j_0}$ and so $\dim(\tilde{Z}\cap Y)=\dim E_{j_0}=k$. Therefore the real algebraic set $\tilde{Z}\cap X$ is of dimension $k$ and is a proper subset of $X$. But this is impossible, because $X$ is $k$-dimensional and irreducible. The contradiction proves that $\mathcal{S}^d(Y)=Y$.

For the final assertion of the proposition take $Z=\bigcup_{j=1}^sZ_j$, where the $Z_j$ are as above. Then $Z\cap X\supset Y$, hence $\dim(Z\cap X)=k$. Thus, by irreducibility of $X$, $Z\cap X=X$ and so $X\subset Z$. By the choice of $Z_j$, $Z$ is of pure (complex) dimension $d$, which is the constant HC dimension of $Y$.
\end{proof}

\begin{corollary}
\label{cor:alg-constant-hcdim}
If $X$ is a pure-dimensional irreducible real-algebraic set in $\C^m$, then $X$ is of constant HC dimension.
Moreover, there is a unique smallest pure-dimensional complex algebraic set $Z$ in $\C^m$ such that $X\subset Z$ and $\dim_\C Z$ is equal to the constant HC dimension of $X$.
\end{corollary}

Proposition~\ref{prop:alg-constant-hcdim} can be, in fact, generalized to the case of Nash real-analytic sets. This is interesting, because an anologous result is false in the transcendental case, as is shown by \cite[Ex.\,6.1]{AS}.

\begin{proposition}
\label{prop:an-constant-hcdim}
Let $R$ be a Nash real-analytic set in $\C^m$ (in the sense of \cite{BCR}). If $R$ is irreducible (as a real-analytic set) and of dimension $k$, then $\overline{\Reg_k(R)}$ is of constant HC dimension.
Moreover, there is a unique smallest pure-dimensional complex-algebraic set $Z$ in $\C^m$ such that $R\subset Z$ and $\dim_\C Z$ is equal to the constant HC dimension of $\overline{\Reg_k(R)}$.
\end{proposition}

\begin{proof}
Set $S:=\overline{\Reg_k(R)}$.
Let $d:=\max\{\dim_{HC}S_p:p\in S\}$, and suppose that $\mathcal{S}^d(S)$ is a proper subset of $S$. Let $E_1,\dots,E_s$ be the connected components of $\Reg_k(S)$. Then $S=\overline{E}_1\cup\dots\cup\overline{E}_s$. As in the proof of Proposition~\ref{prop:alg-constant-hcdim}, for each $j=1,\dots,s$, there is a unique smallest irreducible complex-algebraic set $Z_j$ in $\C^m$ such that $\overline{E}_j\subset Z_j$ and $\dim_\C Z_j=\dim_{HC}(E_j)_p$ for all $p\in E_j$. By our hypothesis, there exists $j_0$ for which $\dim_\C Z_{j_0}<d$.
Set
\[
\tilde{Z}:=\bigcup_{\dim Z_j<d}Z_j\,.
\]
Then $\dim_\C\tilde{Z}<d$, and hence $S\not\subset\tilde{Z}$. Set
\[
\tilde{S}:=\bigcup_{\dim Z_j<d}\overline{E}_j\,.
\]
Let $\tilde{X}$ be the Zariski closure of $\tilde{S}$, and let $X$ be the Zariski closure of $R$. Then $\dim\tilde{X}=\dim\tilde{S}=k$ and $\dim X=\dim R=k$, by \cite[\S\,2.8]{BCR}. Observe that $\tilde{X}$ is a proper subset of $X$. Indeed, $R\not\subset\tilde{X}$, because $S\not\subset\tilde{Z}$ and $\tilde{X}\subset\tilde{Z}$, by construction. On the other hand, $X$ is irreducible: For if $X=X_1\cup X_2$ with $X_1,X_2$ real-algebraic proper subsets of $X$, then $R=(R\cap X_1)\cup(R\cap X_2)$ would be a decomposition of $R$ into proper real-analytic subsets, unless $R\subset X_1$ or $R\subset X_2$ which would contradict the minimality of $X$.
Hence an irreducible real-algebraic set $X$ contains the real-algebraic $\tilde{X}$ as a proper subset and $\dim\tilde{X}=\dim X$, which is impossible. This proves that $\mathcal{S}^d(S)=S$.

The remainder of the proof is the same as that of Proposition~\ref{prop:alg-constant-hcdim}.
\end{proof}


\section{Proof of Theorem~\ref{thm:cr-continuation}}
\label{sec:proof-main}

Given the results of the preceeding section, it is now easy to adapt Shafikov's proof of \cite[Thm.\,1.3]{Sh} to the Nash setting. We shall therefore restrict ourselves to highlighting just the key points of the proof.

First, note that the graph $\Gamma_f$ is semialgebraic, of pure dimension $k:=\dim M$. Moreover, $\Gamma_f$ is $\AR$-irreducible, by \cite[Prop.\,5.1]{Kur} and $\AR$-irreducibility of $M$.
Let $\tilde{M}$ be the subset of $M$ on which $f$ is CR. Since $\mathrm{Reg}_k\Gamma_f$ is dense in $\Gamma_f$, there exists $p\in\tilde{M}$ such that $(p,f(p))\in\mathrm{Reg}_k\Gamma_f$. By \cite[Thm.\,5.2]{Kur}, the map $f$ is analytic outside a set $S_f$ with $\dim S_f\leq\dim M-2$. Therefore the point $p$ can actually be chosen so that $\Gamma_f$ near $(p,f(p))$ is the graph of a smooth map on an open neighbourhood $U^p$ of $p$ in $M$.

If the CR codimension of $M$ is zero, then $M$ is a complex-analytic manifold and $f$ is holomorphic on $U^p$. Then the germ $(\Gamma_f)_{(p,f(p))}$ is complex-analytic, and hence $\Gamma_f$ is an irreducible complex-algebraic set in $\C^m\times\C^n$, by Theorem~\ref{thm:ar-irred-is-C-alg}. Let $\pi:\C^m\times\C^n\to\C^m$ and $\pi':\C^m\times\C^n\to\C^n$ be the projections. Then $\pi|_{\Gamma_f}:\Gamma_f\to M$ is a bijective holomorphic mapping, and hence a biholomorphism (see, e.g., \cite[\S\,3.3, Prop.\,3]{C}). It follows that $f=\pi'|_{\Gamma_f}\circ(\pi|_{\Gamma_f})^{-1}$ is holomorphic on the whole $M$. 
If, in turn, the CR dimension of $M$ is zero then there is nothing to show because any function is CR on $M$.

Suppose then that $M$ is a CR manifold of type $(k,l)$ with both $k$ and $l$ positive.
This part of the proof requires reduction to the case of a generic CR submanifold. It should be observed that, by the proof of \cite[Prop.\,1.4]{AS}, a local embedding of $M$ into $\C^{k+l}$ (which makes $M$ generic) can be chosen semialgebraic. Therefore passing through that embedding does not affect the semialgebraicity of $M$ or $f$. The remainder of the proof follows exactly as in \cite{Sh}, with one major simplification: Namely, as a consequence of Theorem~\ref{thm:ar-irred-const-hcdim}, in our case the graph $\Gamma_f$ has constant HC dimension $m$ and it is contained in an irreducible complex-algebraic subset $Z$ of $\C^m\times\C^n$ of dimension $m$ whose projection to $\C^m$ is generically finite over $M$. Therefore one only needs to remove CR singularities of $f$ over the intersection of $M$ with a complex-algebraic subset $\Sigma$ of $\C^m$, defined as the closure of the projection of the algebraic-constructible set
\[
\mathrm{Sng}{Z}\cup\{z\in\Reg{Z}:\mathrm{rank}\;d_z(\pi|_Z)<m\}\,.\qed
\]

\bibliographystyle{amsplain}

\end{document}